\documentclass{article}
\usepackage{latexsym, amsmath, amssymb}
\usepackage{graphicx}
\usepackage{epstopdf}
\usepackage{amsmath, amsthm, amscd, amsfonts}
\usepackage{amsmath}
\usepackage{array}
\usepackage{multirow}
\usepackage{amssymb}
\usepackage{float}
\usepackage{enumerate}
\usepackage[a4paper, total={6in, 10in}]{geometry}

\newtheorem{thm}{Theorem}[section]

\newtheorem{prop}{Proposition}

\theoremstyle{definition} 
\theoremstyle{question} 
\theoremstyle{remark} %\numberwithin{equation}{section} % MATH -----------------------------------------------------------

\makeatletter
\newcommand*{\rom}[1]{\expandafter\@slowromancap\romannumeral #1@}
\makeatother

\def\bege{\begin{equation}} \def\ende{\end{equation}}

   \def\begr{\begin{eqnarray}}
\def\endr{\end{eqnarray}} 
 
\def\bege{\begin{equation}} \def\ende{\end{equation}}
\def\begr{\begin{eqnarray}} \def\endr{\end{eqnarray}} \def\bnum{\begin{enumerate}} \def\enum{\end{enumerate}}
\begin{document}

\begin{center}
\textbf{Metric and Edge Metric Dimension of Zigzag Edge Coronoid Fused with Starphene}
\end{center}

\begin{center}
Sunny Kumar Sharma$^{1,a}$, Vijay Kumar Bhat$^{1,}$$^{\ast}$, Hassan Raza$^{2,b}$, and Karnika Sharma$^{1,c}$
\end{center}
$^{1}$School of Mathematics, Shri Mata Vaishno Devi University, Katra-$182320$, J \& K, India.\\
$^{2}$Business School, University of Shanghai for Science and Technology, Shanghai 200093, China.\\
$^{a}$sunnysrrm94@gmail.com, $^{\ast}$vijaykumarbhat2000@yahoo.com, $^{b}$hassan\_raza783@yahoo.com,\\
 $^{c}$19dmt001@smvdu.ac.in\\\\
\textbf{Abstract} Let $\Gamma=(V,E)$ be a simple connected graph. $d(\alpha,\epsilon)=min\{d(\alpha, w), d(\alpha, d\}$ computes the distance between a vertex $\alpha \in V(\Gamma)$ and an edge $\epsilon=wd\in E(\Gamma)$. A single vertex $\alpha$ is said to recognize (resolve) two different edges $\epsilon_{1}$ and $\epsilon_{2}$ from $E(\Gamma)$ if $d(\alpha, \epsilon_{2})\neq d(\alpha, \epsilon_{1}\}$. A subset of distinct ordered vertices $U_{E}\subseteq V(\Gamma)$ is said to be an edge metric generator for $\Gamma$ if every pair of distinct edges from $\Gamma$ are recognized by some element of $U_{E}$. An edge metric generator with a minimum number of elements in it, is called an edge metric basis for $\Gamma$. Then, the cardinality of this edge metric basis of $\Gamma$, is called the edge metric dimension of $\Gamma$, denoted by $edim(\Gamma)$. The concept of studying chemical structures using graph theory terminologies is both appealing and practical. It enables chemical researchers to more precisely and easily examine various chemical topologies and networks. In this article, we investigate a fascinating cluster of organic chemistry as a result of this motivation. We consider a zigzag edge coronoid fused with starphene and find its minimum vertex and edge metric generators.
\section{Introduction}
The theory of chemical graphs is the part of graph theory that deals with chemistry and mathematics. The study of various structures related to chemicals from the perspective of graphs is the subject of chemical graph theory. Chemical structures that are complicated and large in size are difficult to examine in their natural state. Then chemical graph theory is employed to make these complex chemical structures understandable. The molecular graph is a graph of a chemical structure in which the atoms are the vertices and the edges reflect the bonds between the atoms.\\

The physical attributes of a chemical structure are studied using a unique mathematical representation in which every atom (vertex) has its own identification or position within the given chemical structure. A few atoms (vertices) are chosen for this unique identification of the whole vertex set so that the set of atoms has a unique location to the selected vertices. This idea is known as metric basis in graph theory \cite{ps} and a resolving set (metric generator) in applied graph theory \cite{fr}. If any element of a metric generator fails (crashes), the entire system can be shut down, to address such problems the concept of fault tolerance in metric generators was introduced by Hernando et al. \cite{hn}.\\

Next, one can think that, rather than obtaining a unique atomic position, bonds could be utilized to shape the given structure, to address this Kelenc et al. \cite{emd} proposed and initiated the study of a new variant of metric dimension in non-trivial connected graphs that focuses on uniquely identifying the graph\textquotesingle s edges, called the edge metric dimension (EMD). Similar to the concept of fault tolerance in resolving sets, the idea of fault tolerance in edge resolving sets has also been introduced by Liu et al. \cite{lft}.\\

The researchers are motivated by the fact that the metric dimension has a variety of practical applications in everyday life and so it has been extensively investigated. Metric dimension is utilized in a wide range of fields of sciences, including robot navigation \cite{krr}, geographical routing protocols \cite{pil}, connected joints in network and chemistry \cite{cee}, telecommunications \cite{zf}, combinatorial optimization \cite{co}, network discovery and verification \cite{zf} etc. NP-hardness and computational complexity for the resolvability parameters are addressed in \cite{1,2}.\\

An organic compound with the chemical formula $C_{6}H_{6}$ is known as benzene. Many commercial, research and industrial operations use it as a solvent. Benzene is a key component of gasoline and can be found in crude oil. Dyes, detergents, resins, plastics, rubber lubricants, medicines, pesticides, and synthetic fibers are all made from it. When benzene rings are linked together, they form larger polycyclic aromatic compounds known as polyacenes.\\

The word coronoid was coined by Brunvoll et al. \cite{edd}, due to its possible relationship with benzenoid. A coronoid is a benzenoid that has a hole in the middle. Coronoid is a polyhex system that has its origin in organic chemistry. The zigzag-edge coronoids, denoted by $HC_{a,b,c}$, as shown in Fig. 1(i), can be considered as a structure obtained by fusing six linear polyacenes segments into a closed loop. This structure is also known as a hollow coronoid \cite{alikom}. Next, starphenes, denoted by $SP_{a,b,c}$, are the two-dimensional polyaromatic hydrocarbons with three polyacene arms joined by a single benzene ring, as shown in Fig. 1(ii). They can be utilized as logic gates in single-molecule electronics. Furthermore, as a type of 2D polyaromatic hydrocarbon, starphenes could be a promising material for organic electronics, such as organic light-emitting diodes (OLEDs) or organic field-effect transistors \cite{111}. A composite benzenoid obtained by fusing a zigzag-edge coronoid $HC_{a,b,c}$ with a starphene $SP_{a,b,c}$ is depicted in Fig. 2. We denote this system by $FCS_{a,b,c}$.\\

The metric dimension was investigated for numerous chemical structures because of several application of this parameter in chemical sciences. \cite{17n} discuss the vertex resolvability of $VC_{5}C_{7}$ and H-Napthalenic nanotubes. \cite{41n} determines the minimum resolving sets for silicates star networks, \cite{40n} set upper bounds for the minimum resolving sets of cellulose network, and \cite{16n} discuss the metric dimension of 2D lattice of Boron nanotubes (Alpha). Similarly, other variants of metric dimension, such as EMD, fault-tolerant metric dimension (FTMD), fault-tolerant edge metric dimension (FTEMD), etc have been studied for different graph families and chemical structures.\\

Azeem and Nadeem \cite{an}, studied metric dimension, EMD, FTMD, and FTEMD for polycyclic aromatic hydrocarbons. Sharma and Bhat \cite{sv, ssv}, studied metric dimension and EMD for some convex polytope graphs. Koam et al. \cite{alikom} studied the metric dimension and FTMD of hollow coronoid structures. The metric dimension and these recently introduced concepts have been studied by many authors for different graph families. For instance, path graphs, cycle graphs, prism graphs, wheel-related graphs, tadpole graphs, cycle with chord graphs, kayak paddle graphs, etc. But still there are several chemical graphs for which the metric dimension and the EMD has not been found yet. Such as the graph $FCS_{a,b,c}$. Thus, this paper aims to compute the metric dimension and the EMD of $FCS_{a,b,c}$.\\

The present paper is organized as follows. In Sect. 2 theory and concepts related to metric dimension, EMD, and independence in their respective metric generators have been discussed. In Sect. 3 we study the metric dimension and independence in the vertex metric generator of $FCS_{a,b,c}$. Sect. 4 gives the edge metric dimension of $FCS_{a,b,c}$. Finally, the conclusion and future work of this paper is presented in section 5.

\section{Preliminaries}

In this section, we discuss some basic concepts, definitions, and existing results related to the metric dimension, edge metric dimension, and independent (vertex and edge) metric generators of graphs.\\\\
Suppose $\Gamma=(V,E)$ is a non-trivial, connected, simple, and finite graph with the edge set $E(\Gamma)$ and the vertex set $V(\Gamma)$. We write $E$ instead of $E(\Gamma)$ and $V$ instead of $V(\Gamma)$ throughout the manuscript when there is no scope for ambiguity. The topological distance (geodesic) between two vertices $a$ and $w$ in $\Gamma$, denoted by $d(a,w)$, is the length of a shortest $a-w$ path between the vertices $a$ and $w$ in $\Gamma$.\\\\
\textbf{Degree of a vertex:} The number of edges that are incident to a vertex of a graph $H$ is known as its degree (or valency) and is denoted by $d_{\alpha}$. The minimum degree and the maximum degree of $\Gamma$ are denoted by $\delta(\Gamma)$ and $\Delta(\Gamma)$, respectively.\\\\
\textbf{Independent set:} \cite{sv} An independent set is a set of vertices in $\Gamma$, in which no two vertices are adjacent.\\\\
\textbf{Metric Dimension:} \cite{ps} If for any three vertices $\alpha$, $\beta$, $\gamma$ $\in V(\Gamma)$, we have $d(\alpha,\beta)\neq d_{G}(\alpha,\gamma)$, then the vertex $\alpha$ is said to recognize (resolve or distinguish) the pair of vertices $\beta$, $\gamma$ $(\beta\neq \gamma)$ in $V(\Gamma)$. If this condition of resolvability is fulfilled by some vertices comprising a subset $U \subseteq V(\Gamma)$ i.e., every pair of different vertices in the given undirected graph $\Gamma$ is resolved by at least one element of $U$, then $U$ is said to be a $metric$ $generator$ ($resolving$ $set$) for $\Gamma$. The $metric$ $dimension$ of the given graph $\Gamma$ is the minimum cardinality of a metric generator $U$, and is usually denoted by $dim(\Gamma)$. The metric generator $U$ with minimum cardinality is the metric basis for $\Gamma$. For an ordered subset of vertices $U=\{a_{1}, a_{2}, a_{3},...,a_{k}\}$, the $k$-code (representation or coordinate) of vertex $j$ in $V(\Gamma)$ is;
\vspace{-7.4mm}
\begin{center}
  \begin{eqnarray*}
  \gamma(j|R_{m})&=&(d(a_{1},j),d(a_{2},j),d(a_{3},j),...,d(a_{k},j))
  \end{eqnarray*}
\end{center}
Then we say that, the set $U$ is a metric generator for $\Gamma$, if $\gamma(a|R_{m})\neq \gamma(w|R_{m})$, for any pair of vertices $a,w \in V(\Gamma)$ with $a\neq w$.\\\\
\textbf{Independent metric generator (IMG):} \cite{ssv}
A set of distinct ordered vertices $U$ in $\Gamma$ is said to be an IMG for $\Gamma$ if $U$ is both independent as well as a metric generator.\\\\
\textbf{Edge Metric Dimension:} \cite{emd} The topological distance between a vertex $a$ and an edge $\epsilon=bw$ is given as $d(a,\epsilon)=min\{d(a,w), d(a,b)\}$. The vertex $\alpha$ is said to recognize (resolve or distinguish) the pair of edges $\epsilon_{1}$, $\epsilon_{2}$ with $\epsilon_{1}\neq \epsilon_{2})$ in $E(\Gamma)$. If this condition of edge resolvability is fulfilled by some vertices comprising a subset $U_{E} \subseteq V(\Gamma)$ i.e., every pair of different edges in the given undirected graph $\Gamma$ is resolved by at least one element of $U_{E}$, then $U_{E}$ is said to be an $edge$ $metric$ $generator$ (EMG) for $\Gamma$. The $edge$ $metric$ $dimension$ of the graph $\Gamma$ is the minimum cardinality of an ERS $U_{E}$, and is usually denoted by $edim(\Gamma)$. The edge metric generator (EMG) $U_{E}$ with minimum cardinality is the edge metric basis (EMB) for $\Gamma$. For an ordered subset of vertices $U_{E}=\{b_{1}, b_{2}, b_{3},...,b_{k}\}$, the $k$-edge code (coordinate) of an edge $\epsilon$ in $E(\Gamma)$ is;
\vspace{-7.4mm}
\begin{center}
  \begin{eqnarray*}
  \gamma_{E}(\epsilon|U_{E})&=&(d(b_{1},\epsilon),d(b_{2},\epsilon),d(b_{3},\epsilon),...,d(b_{k},\epsilon))
  \end{eqnarray*}
\end{center}
Then we say that, the set $U_{E}$ is an EMG for $\Gamma$, if $\gamma_{E}(\epsilon_{1}|U_{E})\neq \gamma(\epsilon_{2}|U_{E})$, for any pair of edges $\epsilon_{1}, \epsilon_{2} \in V(\Gamma)$ with $\epsilon_{1} \neq \epsilon_{2}$.\\\\
\textbf{Independent edge metric generator (IEMG):} \cite{ssv}
A set of distinct vertices $U^{i}_{E}$ (ordered) in $\Gamma$ is said to be an IEMG for $\Gamma$ if $U^{i}_{E}$ is both independent as well as a edge metric generator.\\\\

\begin{center}
  \begin{figure}[h!]
  \centering
   \includegraphics[width=3.5in]{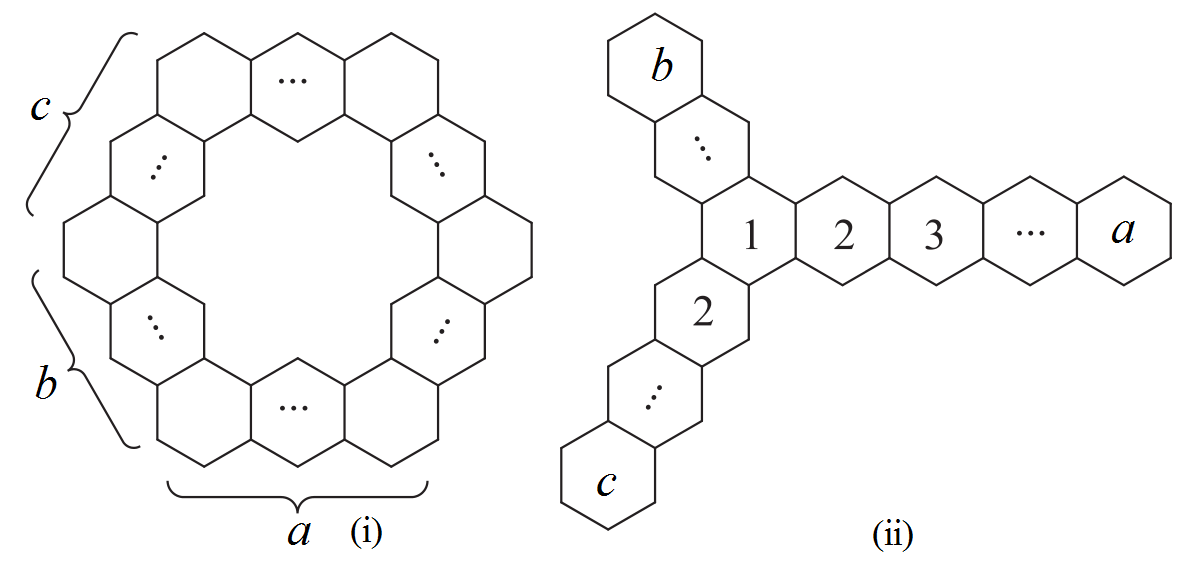}
  \caption{$HC_{a,b,c}$ and $SP_{a,b,c}$}\label{p2}
\end{figure}
\end{center}

$P_{n}$, $C_{n}$, and $K_{n}$ denotes respectively the path graph, cycle graph and the complete graph on $n$ vertices. Then the following results are helpful in obtaining the metric and the edge metric dimension of a graph.
\begin{prop} \cite{emd}
For $n \geq3$, we have $dim(P_{n})= edim(P_{n})=1$, $dim(C_{n})= edim(C_{n})=2$, and $dim(K_{n})= edim(K_{n})=n-1$.
\end{prop}

\section{Metric Dimension of $FCS_{a,b,c}$}
In this section, we obtain the metric dimension and IVMG for $FCS_{a,b,c}$.\\\\
The fused hollow coronoid with starphene structure $FCS_{a,b,c}$ comprises of six sides in which three sides $(a,b,c)$ are symmetric to other three sides $(a,b,c)$ as shown in Fig. 2. This means that $FCS_{a,b,c}$ has three linear polyacenes segments consist of $a$, $b$, and $c$ number of benzene rings. It consists of $3a+3b+3c-11$ number of faces having six sides, two faces having $4a+2b+2c-18$ sides, a face having $4b+4c-18$ sides, and a face having $4a+4b+4c-6$ sides. $FCS_{a,b,c}$ has $6(a+b+c-6)$ number of vertices of degree two and $6(a+b+c-3)$ number of vertices of degree three. From this, we find that $\delta(FCS_{a,b,c})=2$ and $\Delta(FCS_{a,b,c})=3$. The vertex set and the edge set of $FCS_{a,b,c}$, are denoted by $V(FCS_{a,b,c})$ and $E(FCS_{a,b,c})$ respectively. Moreover, the cardinality of edges and vertices in $FCS_{a,b,c}$ is given by $|E(FCS_{a,b,c})|=3(5a+5b+5c-21)$ and $|V(FCS_{a,b,c})|=6(2a+2b+2c-9)$, respectively. The edge and vertex set of $FCS_{a,b,c}$ are describe as follows: $V(FCS_{a,b,c})=\{p_{1,d}, p_{2,d}|1\leq d\leq 2a-1\}\cup\{q_{1,d}, q_{2,d}|1\leq d\leq 2c-1\}\cup\{r_{1,d}, r_{2,d}|1\leq d\leq 2b-1\}\cup \{s_{1,d}, s_{2,d}|1\leq d\leq 2a-3\}\cup \{u_{1,d}, u_{2,d}|1\leq d\leq 2b-3\}\cup\{t_{1,d}, t_{2,d}|1\leq d\leq 2c-3\}\cup \{p_{3,d}, s_{3,d}|1\leq d\leq 2a-5\}\cup \{q_{3,d}, t_{3,d}|1\leq d\leq 2c-5\}\cup \{r_{3,d}, u_{3,d}|1\leq d\leq 2b-5\}$ \\\\
and\\\\
$E(FCS_{a,b,c})=\{p_{1,d}p_{1,d+1}, p_{2,d}p_{2,d+1}|1\leq d\leq 2a-2\}\cup\{q_{1,d}q_{1,d+1}, q_{2,d}q_{2,d+1}|1\leq d\leq 2c-2\}\cup\{r_{1,d}r_{1,d+1}, r_{2,d}r_{2,d+1}|1\leq d\leq 2b-2\}\cup \{s_{1,d}s_{1,d+1}, s_{2,d}s_{2,d+1}|1\leq d\leq 2a-4\}\cup \{u_{1,d}u_{1,d+1}, t_{2,d}t_{2,d+1}|1\leq d\leq 2b-4\}\cup\{t_{1,d}t_{1,d+1}, u_{2,d}u_{2,d+1}|1\leq d\leq 2c-4\}\cup\{p_{1,2d}s_{1,2d-1}, p_{2,2d}s_{2,2d-1}|1 \leq d \leq a-1\}\cup\{q_{1,2d}t_{1,2d-1}, q_{2,2d}u_{2,2d-1}|1 \leq d \leq c-1\}\cup\{r_{1,2d}u_{1,2d-1}, r_{2,2d}t_{2,2d-1}|1 \leq d \leq b-1\}\cup\{p_{3,d}p_{3,d+1}, s_{3,d}s_{3,d+1} \\|1\leq d\leq 2a-6\}\cup\{q_{3,d}q_{3,d+1}, t_{3,d}t_{3,d+1}|1\leq d\leq 2c-6\}\cup\{r_{3,d}r_{3,d+1}, u_{3,d}u_{3,d+1}|1\leq d\leq 2b-6\}\cup\{p_{3,2d-1}s_{3,2d-1}|1 \leq d \leq a-2\}\cup\{q_{3,2d-1}t_{3,2d-1}|1 \leq d \leq c-2\}\cup\{r_{3,2d-1}u_{3,2d-1}|1 \leq d \leq b-2\}\cup\{p_{3,1}r_{3,1}, q_{3,1}u_{3,1}, s_{3,1}t_{3,1}\}\cup\{p_{3,2a-5}t_{2,j-3}, s_{3,2a-5}u_{2,2}, q_{3,2c-5}u_{1,2c-4}, t_{3,2c-5}s_{2,2a-4}, r_{3,2b-5}s_{1,2a-4}, \\u_{3,2b-5}t_{1,2}\}\cup \{p_{1,1}q_{2,1}, s_{1,1}u_{2,1}, p_{1,i}q_{1,1}, s_{1,i-2}t_{1,1}, q_{1,j}r_{1,1}, t_{1,j-2}u_{1,1}, r_{1,k}p_{2,i}, u_{1,k-2}s_{2,l}, p_{2,1}r_{2,k}, \\s_{2,1}t_{2,k-2}, r_{2,1}q_{2,j}, t_{2,1}u_{2,j-2}\}$.\\\\
We name the vertices on the cycle $p_{1,1},...,p_{1,i},q_{1,1},...,q_{1,j}r_{1,1},...,r_{1,k},p_{2,1},...,p_{2,i},q_{2,1},...,q_{2,j}r_{2,1},...,\\r_{2,k}$ as the outer $pqr$-cycle vertices, the vertices on the cycle $s_{1,1},...,s_{1,i-2}, r_{3,k-4},...,r_{3,1}, p_{3,1},...,\\p_{3,i-4}, t_{2,j-2},...,t_{2,1}$ as the vertices of first interior cycle, the vertices on the cycle  $t_{1,1},...,t_{1,j-2}, u_{1,1},\\...,u_{1,k-2}, q_{3,j-4},...,q_{3,1}, u_{3,1},...,u_{3,k-4}$ as the vertices of second interior cycle, and the vertices on the cycle $u_{2,1},...,u_{2,k-2}, s_{2,1},...,s_{2,i-2}, t_{3,j-4},...,t_{3,1}, s_{3,1},...,s_{3,i-4}$ as the vertices of third interior cycle in $FCS_{a,b,c}$. In vertices, $p_{1,i}$, $p_{2,i}$, $q_{1,j}$, $q_{2,j}$, $r_{1,k}$, and $r_{2,k}$, the indices $i=2a-1$, $j=2c-1$ and $k=2b-1$. In the next result, we determine the metric dimension of $FCS_{a,b,c}$.

\begin{thm}
For positive integers $a,b,c\geq 4$, we have $dim(FCS_{a,b,c})=3$.
\end{thm}
\begin{proof}
In order to show that $dim(FCS_{a,b,c})\leq 3$, we construct a metric generator for $FCS_{a,b,c}$. Let $U=\{p_{1,1}, r_{1,1}, r_{2,k}\}$ be a set of distinct vertices from $FCS_{a,b,c}$. We claim that $U$ is a vertex metric generator for $FCS_{a,b,c}$. Now, to obtain $dim(FCS_{a,b,c}) \leq 3$, we can give metric coordinate to every vertex of $FCS_{a,b,c}$ with respect to the set $U$. For the vertices $\{\upsilon=p_{1,d}|1 \leq d \leq 2a-1\}$, the set of vertex metric coordinates is as follow:\\
$P_{1}=\{\gamma(\upsilon|U)=(d-1, 2a+2c-d-1, 2b+2c-2)| d=1\}\cup\{\gamma(\upsilon|U)=(d-1, 2a+2c-d-1, 2b+2c+d-5)|2\leq d \leq 2a-3\}\cup\{\gamma(\upsilon|U)=(d-1, 2a+2c-d-1, 2a+2b+2c-9)|d=2a-2\}\cup\{\gamma(\upsilon|U)=(d-1, 2a+2c-d-1, 2a+2b+2c-8)|d=2a-1\}$.\\\\
For the vertices $\{\upsilon=q_{1,d}|1 \leq d \leq 2c-1\}$, the set of vertex metric coordinates is as follow:\\
$Q_{1}=\{\gamma(\upsilon|U)=(2a+d-2, 2c-d, 2a+2b+2c-7)|d=1\}\cup\{\gamma(\upsilon|U)=(2a+d-2, 2c-d, 2a+2b+2c-8)|d=2\}\cup\{\gamma(\upsilon|U)=(2a+d-2, 2c-d, 2a+2b+2c-d-4)|3\leq d \leq 2c-2\}\cup\{\gamma(\upsilon|U)=(2a+d-2, 2c-d, 2a+2b-1)|d=2c-1\}$.\\\\
For the vertices $\{\upsilon=r_{1,d}|1 \leq d \leq 2b-1\}$, the set of vertex metric coordinates is as follow:\\
$R_{1}=\{\gamma(\upsilon|U)=(2a+2c-2, d-1, 2a+2b-d-1)|d=1\}\cup\{\gamma(\upsilon|U)=(2a+2c+d-5, d-1, 2a+2b-d-1)|2\leq d \leq 2c-3\}\cup\{\gamma(\upsilon|U)=(2a+2b+2c-9, d-1, 2a+2b-d-1)|d=2b-2\}\cup\{\gamma(\upsilon|U)=(2a+2b+2c-8, d-1, 2a+2b-d-1)|d=2b-1\}$.\\\\
For the vertices $\{\upsilon=p_{2,d}|1 \leq d \leq 2a-1\}$, the set of vertex metric coordinates is as follow:\\
$P_{2}=\{\gamma(\upsilon|U)=(d, 2a+2b-d-2, 2b+2c-1)|d=1\}\cup\{\gamma(\upsilon|U)=(d, 2a+2b-d-2, 2b+2c+d-4)|2\leq d \leq 2a-3\}\cup\{\gamma(\upsilon|U)=(d, 2a+2b-d-2, 2a+2b+2c-8)|d=2a-2\}\cup\{\gamma(\upsilon|U)=(d, 2a+2b-d-2, 2a+2b+2c-7)|d=2a-1\}$.\\\\
For the vertices $\{\upsilon=q_{2,d}|1 \leq d \leq 2c-1\}$, the set of vertex metric coordinates is as follow:\\
$Q_{2}=\{\gamma(\upsilon|U)=(d, 2a+2c-1, 2b+2c-d-2)|d=1\}\cup\{\gamma(\upsilon|U)=(d, 2a+2c+d-4, 2b+2c-d-2)|2\leq d \leq 2c-3\}\cup\{\gamma(\upsilon|U)=(d, 2a+2b+2c-8, 2b+2c-d-2)|d=2c-2\}\cup\{\gamma(\upsilon|U)=(d, 2a+2b+2c-7, 2b+2c-d-2)|d=2c-1\}$.\\\\
For the vertices $\{\upsilon=r_{2,d}|1 \leq d \leq 2b-1\}$, the set of vertex metric coordinates is as follow:\\
$R_{2}=\{\gamma(\upsilon|U)=(2c+d-1, 2a+2b+2c-8, 2b-d-1)|d=1\}\cup\{\gamma(\upsilon|U)=(2c+d-1, 2a+2b+2c-9, 2b-d-1)|d=2\}\cup\{\gamma(\upsilon|U)=(2c+d-1, 2a+2b+2c-d-5, 2b-d-1)|3\leq d \leq 2b-2\}\cup\{\gamma(\upsilon|U)=(2c+d-1, 2a+2b-2, 2b-d-1)|d=2b-1\}$.\\\\
For the vertices $\{\upsilon=s_{1,d}|1 \leq d \leq 2a-3\}$, the set of vertex metric coordinates is as follow:\\
$S_{1}=\{\gamma(\upsilon|U)=(d+1, 2a+2c-d-3, 2b+2c+d-5)|1\leq d \leq 2a-5\}\cup\{\gamma(\upsilon|U)=(d+1, 2a+2c-d-3, 2a+2b+2c-11)|d=2a-4\}\cup\{\gamma(\upsilon|U)=(d+1, 2a+2c-d-3, 2a+2b+2c-10)|d=2a-3\}$.\\\\
For the vertices $\{\upsilon=t_{1,d}|1 \leq d \leq 2c-3\}$, the set of vertex metric coordinates is as follow:\\
$T_{1}=\{\gamma(\upsilon|U)=(2a+d-2, 2c-d, 2a+2b+2c-9)|d=1\}\cup\{\gamma(\upsilon|U)=(2a+d-2, 2c-d, 2a+2b+2c-10 )|d=2\}\cup\{\gamma(\upsilon|U)=(2a+d-2, 2c-d, 2a+2b+2c-d-6)|3\leq d \leq 2c-3\}$.\\\\
For the vertices $\{\upsilon=u_{1,d}|1 \leq d \leq 2b-3\}$, the set of vertex metric coordinates is as follow:\\
$U_{1}=\{\gamma(\upsilon|U)=(2a+2c+d-5, d+1, 2a+2b-d-3)|1 \leq d \leq 2b-5\}\cup\{\gamma(\upsilon|U)=(2a+2b+2c-11, d+1, 2a+2b-d-3)|d=2b-4\}\cup \{\gamma(\upsilon|U)=(2a+2b+2c-10, d+1, 2a+2b-d-3)|d=2b-3\}$.\\\\
For the vertices $\{\upsilon=s_{2,d}|1 \leq d \leq 2a-3\}$, the set of vertex metric coordinates is as follow:\\
$S_{2}=\{\gamma(\upsilon|U)=(2b+2c+d-4, 2a+2b-d-4, d+2)|1\leq d \leq 2a-5\}\cup\{\gamma(\upsilon|U)=(2a+2b+2c-10, 2a+2b-d-4, d+2)|d=2a-4\}\cup\{\gamma(\upsilon|U)=(2a+2b+2c-9, 2a+2b-d-4, d+2)|d=2a-3\}$.\\\\
For the vertices $\{\upsilon=t_{2,d}|1 \leq d \leq 2c-3\}$, the set of vertex metric coordinates is as follow:\\
$T_{2}=\{\gamma(\upsilon|U)=(d+2, 2a+2c+d-4, 2b+2c-d-4)|1\leq d \leq 2c-5\}\cup\{\gamma(\upsilon|U)=(d+2, 2a+2b+2c-10, 2b+2c-d-4 )|d=2c-4\}\cup\{\gamma(\upsilon|U)=(d+2, 2a+2b+2c-9, 2b+2c-d-4)|d=2c-3\}$.\\\\
For the vertices $\{\upsilon=u_{2,d}|1 \leq d \leq 2b-3\}$, the set of vertex metric coordinates is as follow:\\
$U_{2}=\{\gamma(\upsilon|U)=(2c+d-1, 2a+2b+2c-10, 2b-d-1)|d=1\}\cup\{\gamma(\upsilon|U)=(2c+d-1, 2a+2b+2c-11, 2b-d-1)|d=2\}\cup \{\gamma(\upsilon|U)=(2c+d-1, 2a+2b+2c-d-7, 2b-d-1)|3 \leq d \leq 2b-3\}$.\\\\
For the vertices $\{\upsilon=p_{3,d}|1 \leq d \leq 2a-5\}$, the set of vertex metric coordinates is as follow:\\
$P_{3}=\{\gamma(\upsilon|U)=(2a+2c-d-6, 2a+2c+d-6, 2a+2b-d-4)|1\leq d \leq 2a-5\}$.\\\\
For the vertices $\{\upsilon=q_{3,d}|1 \leq d \leq 2c-5\}$, the set of vertex metric coordinates is as follow:\\
$Q_{3}=\{\gamma(\upsilon|U)=(2a+2b+d-7, 2b+2c-d-7, 2a+2c-d-5)|1\leq d \leq 2c-5\}$.\\\\
For the vertices $\{\upsilon=r_{3,d}|1 \leq d \leq 2b-5\}$, the set of vertex metric coordinates is as follow:\\
$R_{3}=\{\gamma(\upsilon|U)=(2a+2b-d-7, 2b+2c-d-5, 2a+2b+d-7)|1\leq d \leq 2b-5\}$.\\\\
For the vertices $\{\upsilon=s_{3,d}|1 \leq d \leq 2a-5\}$, the set of vertex metric coordinates is as follow:\\
$S_{3}=\{\gamma(\upsilon|U)=(2a+2c-d-5, 2a+2c+d-7, 2a+2b-d-5)|1\leq d \leq 2a-5\}$.\\\\
For the vertices $\{\upsilon=t_{3,d}|1 \leq d \leq 2c-5\}$, the set of vertex metric coordinates is as follow:\\
$T_{3}=\{\gamma(\upsilon|U)=(2a+2c+d-6, 2b+2c-d-6, 2a+2c-d-6)|1\leq d \leq 2c-5\}$.\\\\
For the vertices $\{\upsilon=u_{3,d}|1 \leq d \leq 2b-5\}$, the set of vertex metric coordinates is as follow:\\
$U_{3}=\{\gamma(\upsilon|U)=(2a+2b-d-6, 2b+2c-d-6, 2a+2b+d-6)|1\leq d \leq 2b-5\}$.\\\\
Now, from these sets of vertex metric codes for the graph $FCS_{a,b,c}$, we find that $|P_{1}|=|P_{2}|=2a-1$, $|Q_{1}|=|Q_{2}|=2c-1$, $|R_{1}|=|R_{2}|=2b-1$, $|S_{1}|=|S_{2}|=2a-3$, $|T_{1}|=|T_{2}|=2c-3$, $|U_{1}|=|U_{2}|=2b-3$, $|P_{3}|=|S_{3}|=2a-5$, $|Q_{3}|=|T_{3}|=2c-5$, and $|R_{3}|=|U_{3}|=2b-5$. We see that the sum of all of these cardinalities is equal to $|V(FCS_{a,b,c})|$ and which is $6(2a+2b+2c-9)$. Moreover, all of these sets are pairwise disjoint, which implies that $dim(FCS_{a,b,c})\leq 3$. To complete the proof, we have to show that $dim(FCS_{a,b,c})\geq 3$. To show this, we have to prove that there exists no vertex metric generator $U$ for $FCS_{a,b,c}$ such that
$|U|\leq2$. Since, the graph $FCS_{a,b,c}$ is not a path graph, so the possibility of a singleton vertex metric generator for $FCS_{a,b,c}$ is ruled out \cite{cee}. Next, suppose on the contrary that there exists an edge resolving set $U$ with $|U|=2$. Therefore, we have the following ten cases to be discussed (for the contradictions, the naturals $a$, $b$, and $c$ are $\geq5$):

\begin{center}
  \begin{figure}[h!]
  \centering
   \includegraphics[width=3.2in]{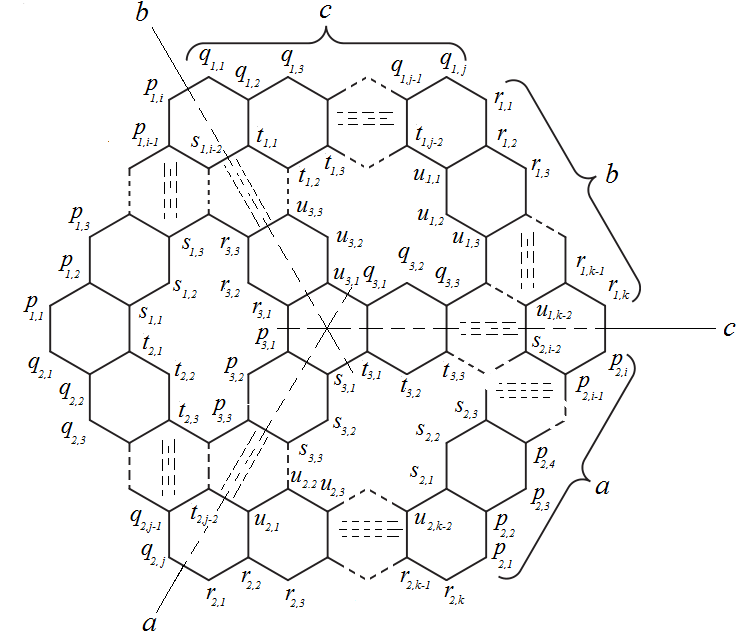}
  \caption{$FCS_{a,b,c}$}\label{p2}
\end{figure}
\end{center}
\textbf{Case(\rom{1})} When $U=\{a, b\}$, where $a$ and $b$ are the vertices from the outer $pqr$-cycle of $FCS_{a,b,c}$.
\begin{itemize}
  \item Suppose $U=\{p_{1,1}, p_{1,d}\}$, $p_{1,d}$ ($2\leq d \leq i$). Then $\gamma(q_{1,2}|U)=\gamma(t_{1,2} |U)$, for $2\leq d\leq i-1$; $\gamma(u_{1,2}|U)=\gamma(r_{1,2}|U)$ when $d=i$, a contradiction.
  \item Suppose $U=\{p_{1,1}, q_{1,d}\}$, $q_{1,d}$ ($1\leq d \leq j$). Then $\gamma(r_{1,2}|U)=\gamma(u_{1,2} |U)$, for $d=1$; $\gamma(p_{1,i}|U)=\gamma(s_{1,i-2}|U)$ when $2\leq d\leq j$, a contradiction.
  \item Suppose $U=\{p_{1,1}, r_{1,d}\}$, $r_{1,d}$ ($1\leq d \leq k$). Then $\gamma(p_{1,i}|U)=\gamma(s_{1,i-2}|U)$ for $d=1$; $\gamma(p_{1,3}|U) =\gamma(s_{1,1}|U)$, for $2\leq d\leq k$, a contradiction.
  \item Suppose $U=\{p_{1,1}, p_{2,d}\}$, $p_{2,d}$ ($1\leq d \leq i$). Then $\gamma(s_{1,1}|U)=\gamma(q_{2,2}|U)$ for $1\leq d\leq i-3$; $\gamma(p_{1,3}|U)=\gamma(s_{1,1}|U)$, for $i-2\leq d\leq i$, a contradiction.
  \item Suppose $U=\{p_{1,1}, r_{2,d}\}$, $r_{2,d}$ ($1\leq d \leq k$). Then $\gamma(s_{2,2}|U)=\gamma(p_{2,2}|U)$ for $d=1$; $\gamma(q_{2,j}|U)=\gamma(t_{2,j-2}|U)$ for $2\leq d\leq k$, a contradiction.
  \item Suppose $U=\{p_{1,1}, q_{2,d}\}$, $q_{2,d}$ ($1\leq d \leq j$). Then $\gamma(s_{2,2}|U)=\gamma(p_{2,2}|U)$ for $1\leq d\leq j$, a contradiction.
  \end{itemize}

\textbf{Case(\rom{2})} When $U=\{a, b\}$, where $a$ and $b$ are the vertices from the first interior cycle of $FCS_{a,b,c}$.

\begin{itemize}
  \item Suppose $U=\{s_{1,1}, s_{1,d}\}$, $s_{1,d}$ ($2\leq d \leq i-2$). Then $\gamma(q_{2,2}|U)=\gamma(t_{2,2}|U)$, for $2\leq d\leq i-1$, a contradiction.
  \item Suppose $U=\{s_{1,1}, r_{3,d}\}$, $r_{3,d}$ ($1\leq d \leq k-5$). Then $\gamma(p_{1,2}|U)=\gamma(t_{2,1}|U)$, for $1\leq d\leq k-5$, a contradiction.
  \item Suppose $U=\{s_{1,1}, t_{2,d}\}$, $t_{2,d}$ ($1\leq d \leq j-2$). Then $\gamma(p_{1,2}|U)=\gamma(s_{1,2}|U)$, for $1\leq d\leq j-2$, a contradiction.
  \item Suppose $U=\{s_{1,1}, p_{3,d}\}$, $p_{3,d}$ ($1\leq d \leq i-5$). Then $\gamma(s_{2,2}|U)=\gamma(p_{2,2}|U)$, for $1\leq d\leq 2a-3$, a contradiction.
  \end{itemize}

\textbf{Case(\rom{3})} When $U=\{a, b\}$, where $a$ and $b$ are the vertices from the second interior cycle of $FCS_{a,b,c}$.

\begin{itemize}
  \item Suppose $U=\{u_{1,1}, u_{1,d}\}$, $u_{1,d}$ ($1\leq d \leq k-2$). Then $\gamma(p_{2,i-1}|U)=\gamma(s_{2,i-3}|U)$, for $1\leq d\leq k-2$, a contradiction.
  \item Suppose $U=\{u_{1,1}, t_{1,d}\}$, $t_{1,d}$ ($1\leq d \leq j-2$). Then $\gamma(r_{1,1}|U)=\gamma(r_{1,3}|U)$, for $1\leq d\leq j-2$, a contradiction.
  \item Suppose $U=\{u_{1,1}, u_{3,d}\}$, $u_{3,d}$ ($1\leq d \leq k-5$). Then $\gamma(p_{1,i}|U)=\gamma(p_{1,i-2}|U)$, for $1\leq d\leq k-5$, a contradiction.
  \item Suppose $U=\{u_{1,1}, q_{3,d}\}$, $q_{3,d}$ ($1\leq d \leq j-5$). Then $\gamma(p_{1,i}|U)=\gamma(p_{1,i-2}|U)$, for $1\leq d\leq j-5$, a contradiction.
\end{itemize}

\textbf{Case(\rom{4})} When $U=\{a, b\}$, where $a$ and $b$ are the vertices from the third interior cycle of $FCS_{a,b,c}$.

\begin{itemize}
  \item Suppose $U=\{s_{2,1}, s_{2,d}\}$, $s_{2,d}$ ($2\leq d \leq i-2$). Then $\gamma(r_{2,k}|U)=\gamma(r_{2,k-2}|U)$, for $2\leq d\leq i-2$, a contradiction.
  \item Suppose $U=\{s_{2,1}, u_{2,d}\}$, $u_{2,d}$ ($1\leq d \leq k-2$). Then $\gamma(p_{2,1}|U)=\gamma(p_{2,3}|U)$, for $1\leq d\leq k-2$, a contradiction.
  \item Suppose $U=\{s_{2,1}, s_{3,d}\}$, $s_{3,d}$ ($1\leq d \leq i-5$). Then $\gamma(q_{2,j}|U)=\gamma(q_{2,j-2}|U)$, for $1\leq d\leq i-5$, a contradiction.
  \item Suppose $U=\{s_{2,1}, t_{3,d}\}$, $t_{3,d}$ ($1\leq d \leq j-5$). Then $\gamma(q_{2,j}|U)=\gamma(q_{2,j-2}|U)$, for $1\leq d\leq j-5$, a contradiction.
\end{itemize}

\textbf{Case(\rom{5})} When $U=\{a, b\}$, where $a$ is in outer $pqr$-cycle and $b$ is in the first interior cycle of $FCS_{a,b,c}$.

\begin{itemize}
  \item Suppose $U=\{p_{1,1}, s_{1,d}\}$, $s_{1,d}$ ($1\leq d \leq i-2$). Then $\gamma(r_{2,2}|U)=\gamma(u_{2,2}|U)$, for $1\leq d\leq i-2$, a contradiction.
  \item Suppose $U=\{p_{1,1}, t_{2,d}\}$, $t_{2,d}$ ($1\leq d \leq j-2$). Then $\gamma(r_{2,2}|U)=\gamma(u_{2,2}|U)$, for $1\leq d\leq j-2$, a contradiction.
  \item Suppose $U=\{p_{1,1}, r_{3,d}\}$, $r_{3,d}$ ($1\leq d \leq k-5$). Then $\gamma(q_{1,4}|U)=\gamma(t_{1,4}|U)$, for $1\leq d\leq k-5$, a contradiction.
  \item Suppose $U=\{p_{1,1}, p_{3,d}\}$, $p_{3,d}$ ($1\leq d \leq i-5$). Then $\gamma(q_{1,4}|U)=\gamma(t_{1,4}|U)$, for $1\leq d\leq i-5$, a contradiction.
\end{itemize}

\textbf{Case(\rom{6})} When $U=\{a, b\}$, where $a$ is in outer $pqr$-cycle and $b$ is in the second interior cycle of $FCS_{a,b,c}$.

\begin{itemize}
  \item Suppose $U=\{p_{1,1}, t_{1,d}\}$, $t_{1,d}$ ($1\leq d \leq j-2$). Then $\gamma(u_{1,2}|U)=\gamma(r_{1,2}|U)$, for $1\leq d\leq j-2$, a contradiction.
  \item Suppose $U=\{p_{1,1}, u_{1,d}\}$, $u_{1,d}$ ($1\leq d \leq k-2$). Then $\gamma(t_{1,1}|U)=\gamma(u_{3,k-5}|U)$, for $1\leq d\leq k-2$, a contradiction.
  \item Suppose $U=\{p_{1,1}, u_{3,d}\}$, $u_{3,d}$ ($1\leq d \leq k-5$). Then $\gamma(q_{1,4}|U)=\gamma(t_{1,4}|U)$, for $1\leq d\leq k-5$, a contradiction.
  \item Suppose $U=\{p_{1,1}, q_{3,d}\}$, $q_{3,d}$ ($1\leq d \leq j-5$). Then $\gamma(u_{3,2}|U)=\gamma(r_{3,1}|U)$, for $1\leq d\leq j-5$, a contradiction.
\end{itemize}

\textbf{Case(\rom{7})} When $U=\{a, b\}$, where $a$ is in outer $pqr$-cycle and $b$ is in the third interior cycle of $FCS_{a,b,c}$.

\begin{itemize}
  \item Suppose $U=\{p_{1,1}, u_{2,d}\}$, $u_{2,d}$ ($1\leq d \leq k-2$). Then $\gamma(s_{2,2}|U)=\gamma(p_{2,2}|U)$, for $1\leq d\leq k-2$, a contradiction.
  \item Suppose $U=\{p_{1,1}, s_{3,d}\}$, $s_{3,d}$ ($1\leq d \leq i-5$). Then $\gamma(u_{3,4}|U)=\gamma(r_{2,4}|U)$, for $1\leq d\leq i-5$, a contradiction.
  \item Suppose $U=\{p_{1,1}, s_{2,d}\}$, $s_{2,d}$ ($1\leq d \leq i-2$). Then $\gamma(s_{1,1}|U)=\gamma(q_{2,2}|U)$, for $1\leq d\leq i-5$;
   $\gamma(s_{3,2}|U)=\gamma(p_{3,1}|U)$, for $i-4\leq d\leq i-2$, a contradiction.
  \item Suppose $U=\{p_{1,1}, t_{3,d}\}$, $t_{3,d}$ ($1\leq d \leq j-5$). Then $\gamma(s_{3,2}|U)=\gamma(p_{3,1}|U)$, for $1\leq d\leq j-5$, a contradiction.
\end{itemize}

\textbf{Case(\rom{8})} When $U=\{a, b\}$, where $a$ is in first interior cycle and $b$ is in the second interior cycle of $FCS_{a,b,c}$.

\begin{itemize}
  \item Suppose $U=\{s_{1,1}, t_{1,d}\}$, $t_{1,d}$ ($1\leq d \leq j-2$). Then $\gamma(t_{2,2}|U)=\gamma(q_{2,2}|U)$, for $1\leq d\leq j-2$, a contradiction.
  \item Suppose $U=\{s_{1,1}, u_{3,d}\}$, $u_{3,d}$ ($1\leq d \leq k-5$). Then $\gamma(p_{1,2}|U)=\gamma(t_{2,1}|U)$, for $1\leq d\leq k-5$, a contradiction.
  \item Suppose $U=\{s_{1,1}, u_{1,d}\}$, $u_{1,d}$ ($1\leq d \leq k-2$). Then $\gamma(s_{3,2}|U)=\gamma(p_{3,1}|U)$, for $1\leq d\leq k-2$, a contradiction.
  \item Suppose $U=\{s_{1,1}, q_{3,d}\}$, $q_{3,d}$ ($1\leq d \leq j-5$). Then $\gamma(s_{3,2}|U)=\gamma(p_{3,1}|U)$, for $1\leq d\leq j-5$, a contradiction.
\end{itemize}

\textbf{Case(\rom{9})} When $U=\{a, b\}$, where $a$ is in first interior cycle and $b$ is in the third interior cycle of $FCS_{a,b,c}$.

\begin{itemize}
  \item Suppose $U=\{s_{1,1}, u_{2,d}\}$, $u_{2,d}$ ($1\leq d \leq k-2$). Then $\gamma(p_{1,1}|U)=\gamma(p_{1,3}|U)$, for $1\leq d\leq k-2$, a contradiction.
  \item Suppose $U=\{s_{1,1}, s_{2,d}\}$, $s_{2,d}$ ($1\leq d \leq i-2$). Then $\gamma(u_{3,2}|U)=\gamma(r_{3,1}|U)$, for $1\leq d\leq i-2$, a contradiction.
  \item Suppose $U=\{s_{1,1}, t_{3,d}\}$, $t_{3,d}$ ($1\leq d \leq j-5$). Then $\gamma(u_{3,2}|U)=\gamma(r_{3,1}|U)$, for $1\leq d\leq j-5$, a contradiction.
  \item Suppose $U=\{s_{1,1}, s_{3,d}\}$, $s_{3,d}$ ($1\leq d \leq i-5$). Then $\gamma(r_{2,4}|U)=\gamma(u_{2,4}|U)$, for $1\leq d\leq i-5$, a contradiction.
\end{itemize}

\textbf{Case(\rom{10})} When $U=\{a, b\}$, where $a$ is in second interior cycle and $b$ is in the third interior cycle of $FCS_{a,b,c}$.

\begin{itemize}
  \item Suppose $U=\{u_{1,1}, u_{2,d}\}$, $u_{2,d}$ ($1\leq d \leq k-2$). Then $\gamma(r_{3,3}|U)=\gamma(u_{3,2}|U)$, for $1\leq d\leq k-2$, a  contradiction.
  \item Suppose $U=\{u_{1,1}, s_{2,d}\}$, $s_{2,d}$ ($1\leq d \leq i-2$). Then $\gamma(q_{1,j-1}|U)=\gamma(t_{1,j-3}|U)$, for $1\leq d\leq i-2$, a contradiction.
  \item Suppose $U=\{u_{1,1}, t_{3,d}\}$, $t_{3,d}$ ($1\leq d \leq j-5$). Then $\gamma(s_{3,2}|U)=\gamma(p_{3,1}|U)$, for $1\leq d\leq j-5$, a contradiction.
  \item Suppose $U=\{u_{1,1}, s_{3,d}\}$, $s_{3,d}$ ($1\leq d \leq i-5$). Then $\gamma(t_{3,2}|U)=\gamma(q_{3,1}|U)$, for $1\leq d\leq i-5$, a contradiction.
\end{itemize}

As a result, we infer that for $FCS_{a,b,c}$, there is no vertex metric generator $U$ such that $|U|=2$. Therefore, we must have $|U| \geq 3$ i.e., $dim(FCS_{a,b,c})\geq 3$. Hence, $dim(FCS_{a,b,c})=3$, which concludes the theorem.
\end{proof}
In terms of minimum an IVMG, we have the following result

\begin{thm}
For $a,b,c\geq4$, the graph $FCS_{a,b,c}$ has an IVMG with cardinality three.
\end{thm}

\begin{proof}
To show that, for zigzag edge coronoid fused with starphene $FCS_{a,b,c}$, there exists an IVMG $U^{i}$ with $|U^{i}|=3$, we follow the same technique as used in Theorem $1$.\\\\
Suppose $U^{i} = \{p_{1,1}, r_{1,1}, r_{2, k}\} \subset V(FCS_{a,b,c})$. Now, by using the definition of an independent set and following the same pattern as used in Theorem $1$, it is simple to show that the set of vertices $U^{i}= \{p_{1,1}, r_{1,1}, r_{2, k}\}$ forms an IVMG for $FCS_{a,b,c}$ with $|U^{i}|=3$, which concludes the theorem. \\
\end{proof}

\section{Edge Metric Dimension of $FCS_{a,b,c}$}
In this section, we obtain the metric dimension and IVMG for $FCS_{a,b,c}$.\\\\

\begin{thm}
For positive integers $a,b,c\geq 4$, we have $edim(FCS_{a,b,c})=3$.
\end{thm}
\begin{proof}
In order to show that $edim(FCS_{a,b,c}) \leq 3$, we construct an edge metric generator for $FCS_{a,b,c}$. Let $U_{E}=\{p_{1,1}, r_{1,1}, r_{2, k}\}$ be a set of distinct vertices from $FCS_{a,b,c}$. We claim that $U_{E}$ is an edge metric generator for $FCS_{a,b,c}$. Now, to obtain $edim(FCS_{a,b,c}) \leq 3$, we can give edge metric coordinate to every edge of $FCS_{a,b,c}$ with respect to $U_{E}$. For the edges $\{\eta=p_{1,d}p_{1,d+1}|1 \leq d \leq 2a-2\}$, the set of edge metric coordinates is as follow:\\
$P_{1}=\{\gamma_{E}(\eta|U_{E})=(d-1, 2a+2c-d-2, 2b+2-3)|d=1\}\cup \{\gamma_{E}(\eta|U_{E})=(d-1, 2a+2c-d-2, 2b+2c+d-5)|2\leq d \leq 2a-4\}\cup \{\gamma_{E}(\eta|U_{E})=(d-1, 2a+2c-d-2, 2a+4b-8)|2a-3 \leq d \leq 2a-2\}$.\\\\
For the edges $\{\eta=q_{1,d}q_{1,d+1}|1 \leq d \leq 2c-2\}$, the set of edge metric coordinates is as follow:\\
$Q_{1}=\{\gamma_{E}(\eta|U_{E})=(2a+d-2, 2c-d-1, 2a+2b+2c-8)|1 \leq d \leq 2\}\cup\{\gamma_{E}(\eta|U_{E})=(2a+d-2, 2c-d-1, 2a+2b+2c-d-5)|3 \leq d \leq 2c-3\}\cup\{\gamma_{E}(\eta|U_{E})=(2a+d-2, 2c-d-1, 2a+2b-2)|d=2c-2\}$.\\\\
For the edges $\{\eta=r_{1,d}r_{1,d+1}|1 \leq d \leq 2b-2\}$, the set of edge metric coordinates is as follow:\\
$R_{1}=\{\gamma_{E}(\eta|U_{E})=(2a+2c-3, d-1, 2a+2b-d-1)|d=1\}\cup\{\gamma_{E}(\eta|U_{E})=(2a+2c+d-5, d-1, 2a+2b-d-1)|2 \leq d \leq 2b-4\}\cup\{\gamma_{E}(\eta|U_{E})=(2a+2b+2c-9, d-1, 2a+2b-d-1)|2b-3 \leq d \leq 2b-2\}$.\\\\
For the edges $\{\eta=p_{2,d}p_{2,d+1}|1 \leq d \leq 2a-2\}$, the set of edge metric coordinates is as follow:\\
$P_{2}=\{\gamma_{E}(\eta|U_{E})=(2b+2c-2, 2a+2b-d-1, d)|d=1\}\cup \{\gamma_{E}(\eta|U_{E})=(2b+2c+d-4, 2a+2b-d-1, d)|2 \leq d \leq 2a-4\}\cup \{\gamma_{E}(\eta|U_{E})=(2a+2b+2c-8, 2a+2b-d-1, d)| 2a-3 \leq d \leq 2a-2\}$.\\\\
For the edges $\{\eta=q_{2,d}q_{2,d+1}|1 \leq d \leq 2c-2\}$, the set of edge metric coordinates is as follow:\\
$Q_{2}=\{\gamma_{E}(\eta|U_{E})=(d, 2a+2c-2, 2b+2c-d-3)|d=1\}\cup \{\gamma_{E}(\eta|U_{E})=(d, 2a+2c+d-4, 2b+2c-d-3))|2 \leq d \leq 2c-4\}\cup \{\gamma_{E}(\eta|U_{E})=(d, 2a+2b+2c-8, 2b+2c-d-3))|2c-3 \leq d \leq 2c-2\}$.\\\\
For the edges $\{\eta=r_{2,d}r_{2,d+1}|1 \leq d \leq 2b-2\}$, the set of edge metric coordinates is as follow:\\
$R_{2}=\{\gamma_{E}(\eta|U_{E})=(2c+d-1, 2a+2b+2c-9, 2b-d-2)|1\leq d \leq 2\}\cup\{\gamma_{E}(\eta|U_{E})=(2c+d-1, 2a+4b-d-6, 2b-d-2)|3 \leq d \leq 2b-3\}\cup \{\gamma_{E}(\eta|U_{E})=(2c+d-1, 2a+2b-3, 2b-d-2)|d=2b-2\}$.\\\\
For the edges $\{\eta=s_{1,d}s_{1,d+1}|1 \leq d \leq 2a-4\}$, the set of edge metric coordinates is as follow:\\
$S_{1}=\{\gamma_{E}(\eta|U_{E})=(d+1, 2a+2c-d-4, 2b+2c+d-5)|1 \leq d \leq 2a-6\}\cup \{\gamma_{E}(\eta|U_{E})=(d+1, 2a+2c-d-4, 2a+4b-11)|2a-5 \leq d \leq 2a-4\}$.\\\\
For the edges $\{\eta=t_{1,d}t_{1,d+1}|1 \leq d \leq 2c-4\}$, the set of edge metric coordinates is as follow:\\
$T_{1}=\{\gamma_{E}(\eta|U_{E})=(2a+d, 2c-d-1, 2a+2b+2c-10)|1 \leq d \leq 2\}\{\gamma_{E}(\eta|U_{E})=(2a+d, 2c-d-1, 2a+2b+2c-d-7)|3 \leq d \leq 2c-4\}$.\\\\
For the edges $\{\eta=u_{1,d}u_{1,d+1}|1 \leq d \leq 2b-4\}$, the set of edge metric coordinates is as follow:\\
$U_{1}=\{\gamma_{E}(\eta|U_{E})=(2a+2c+d-5, d+1, 2a+2b-d-4)|1 \leq d \leq 2b-6\}\cup\{\gamma_{E}(\eta|U_{E})=(2a+2b+2c-11, d+1, 2a+2b-d-4)|2b-3 \leq d \leq 2b-4\}$.\\\\
For the edges $\{\eta=s_{2,d}s_{2,d+1}|1 \leq d \leq 2a-4\}$, the set of edge metric coordinates is as follow:\\
$S_{2}=\{\gamma_{E}(\eta|U_{E})=(2b+2c+d-4, 2a+2b-d-5, d+2)|1 \leq d \leq 2a-6\}\cup\{\gamma_{E}(\eta|U_{E})=(2a+2b+2c-10, 2a+2b-d-5, d+2)|2a-5 \leq d \leq 2a-4\}$.\\\\
For the edges $\{\eta=t_{2,d}t_{2,d+1}|1 \leq d \leq 2c-4\}$, the set of edge metric coordinates is as follow:\\
$T_{2}=\{\gamma_{E}(\eta|U_{E})=(d+2, 2a+2c+d-4, 2b+2c-d-7)|1 \leq d \leq 2b-6\}\cup\{\gamma_{E}(\eta|U_{E})=(d+2, 2a+2b+2c-10, 2b+2c-d-7)|2b-5 \leq d \leq 2c-4\}$.\\\\
For the edges $\{\eta=u_{2,d}u_{2,d+1}|1 \leq d \leq 2b-4\}$, the set of edge metric coordinates is as follow:\\
$U_{2}=\{\gamma_{E}(\eta|U_{E})=(2c+d-1, 2a+2b+2c-11, 2b-d-2)|1 \leq d \leq 2\}\cup\{\gamma_{E}(\eta|U_{E})=(2c+d-1, 2a+4b-d-8, 2b-d-2)|3 \leq d \leq 2b-4\}$.\\\\
For the edges $\{\eta=p_{3,d}p_{3,d+1}|1 \leq d \leq 2a-6\}$, the set of edge metric coordinates is as follow:\\
$P_{3}=\{\gamma_{E}(\eta|U_{E})=(2a+2c-d-7, 2b+2c+d-6, 2a+2b-d-7)|1 \leq d \leq 2a-6\}$.\\\\
For the edges $\{\eta=q_{3,d}q_{3,d+1}|1 \leq d \leq 2c-6\}$, the set of edge metric coordinates is as follow:\\
$Q_{3}=\{\gamma_{E}(\eta|U_{E})=(2a+2b+d-7, 2b+2c-d-8, 2a+2c-d-6)|1 \leq d \leq 2c-6\}$.\\\\
For the edges $\{\eta=r_{3,d}r_{3,d+1}|1 \leq d \leq 2b-6\}$, the set of edge metric coordinates is as follow:\\
$R_{3}=\{\gamma_{E}(\eta|U_{E})=(2a+2b-d-8, 2b+2c-d-6, 2a+2b+d-7)|1 \leq d \leq 2b-6\}$.\\\\
For the edges $\{\eta=s_{3,d}s_{3,d+1}|1 \leq d \leq 2a-6\}$, the set of edge metric coordinates is as follow:\\
$S_{3}=\{\gamma_{E}(\eta|U_{E})=(2a+2c-d-6, 2b+2c+d-7, 2a+2b-d-8)|1 \leq d \leq 2a-6\}$.\\\\
For the edges $\{\eta=t_{3,d}t_{3,d+1}|1 \leq d \leq 2c-6\}$, the set of edge metric coordinates is as follow:\\
$T_{3}=\{\gamma_{E}(\eta|U_{E})=(2a+2b+d-6, 2b+2c-d-7, 2a+2c-d-7)|1 \leq d \leq 2c-6\}$.\\\\
For the edges $\{\eta=u_{3,d}u_{3,d+1}|1 \leq d \leq 2b-6\}$, the set of edge metric coordinates is as follow:\\
$U_{3}=\{\gamma_{E}(\eta|U_{E})=(2a+2b-d-7, 2b+2c-d-7, 2a+2b+d-6)|1 \leq d \leq 2b-6\}$.\\\\
For the edges $\{\eta_{1}=p_{1,i}q_{1,1}, \eta_{2}=s_{1,i-2}t_{1,1}, \eta_{3}=r_{1,1}q_{1,j}, \eta_{4}=t_{1,j}u_{1,1}, \eta_{5}=r_{1,k}p_{2,i}, \eta_{6}=u_{1,k-2}s_{2,i-2}, \eta_{7}=p_{2,1}r_{2,k}, \eta_{8}=s_{2,1}u_{2,k-2}, \eta_{9}=r_{2,1}q_{2,j}, \eta_{10}=u_{2,1}t_{2,j-2}, \eta_{11}=p_{1,1}q_{2,1}, \eta_{12}=u_{1,1}t_{2,1}\}$, the set of edge metric coordinates is as follow:\\
$V_{1}=\{\gamma_{E}(\eta_{1}|U_{E})=(2a-2, 2c-1, 2a+4b-8), \gamma_{E}(\eta_{2}|U_{E})=(2a-2, 2c-1, 2a+4b-10), \gamma_{E}(\eta_{3}|U_{E})=(2a+2c-3, 0, 2a+2b-2), \gamma_{E}(\eta_{4}|U_{E})=(2a+2c-5, 2, 2a+2b-4), \gamma_{E}(\eta_{5}|U_{E})=(2a+2b+2c-10, 2b-2, 2a-1), \gamma_{E}(\eta_{6}|U_{E})=(2a+2b+2c-8, 2b-2, 2a-1), \gamma_{E}(\eta_{7}|U_{E})=(2b+2c-2, 2a+2b-3, 0), \gamma_{E}(\eta_{8}|U_{E})=(2b+2c-4, 2a+2b-5, 2), \gamma_{E}(\eta_{9}|U_{E})=(2c-1, 2a+2b+2c-10, 2b-2), \gamma_{E}(\eta_{10}|U_{E})=(2c-1, 2a+2b+2c-8, 2b-4), \gamma_{E}(\eta_{11}|U_{E})=(0, 2a+2c-2, 2a+2c-3), \gamma_{E}(\eta_{12}|U_{E})=(2, 2a+2c-4, 2b+2c-5)\}$.\\\\
For the edges $\{\eta=p_{1,2d}s_{1,2d-1}|1 \leq d \leq a-1\}$, the set of edge metric coordinates is as follow:\\
$PS_{1}=\{\gamma_{E}(\eta|U_{E})=(2d-1, 4a+2c-2d-12, 4a+2c+2d-18)|1 \leq d \leq a-2\}\cup\{\gamma_{E}(\eta|U_{E})=(2d-1, 4a+2c-2d-12, 2a+4b-10)|d=a-1 \}$.\\\\
For the edges $\{\eta=q_{1,2d}t_{1,2d-1}|1 \leq d \leq c-1\}$, the set of edge metric coordinates is as follow:\\
$QT_{1}=\{\gamma_{E}(\eta|U_{E})=(4a+2d-13, 4c-2d-8, 2a+2b+2c-9)|d=1\}\cup\{\gamma_{E}(\eta|U_{E})=(4a+2d-13, 4c-2d-8, 4a+2b+2c-2d-15)|1 \leq d \leq c-1\}$.\\\\
For the edges $\{\eta=r_{1,2d}u_{1,2d-1}|1 \leq d \leq b-1\}$, the set of edge metric coordinates is as follow:\\
$RU_{1}=\{\gamma_{E}(\eta|U_{E})=(4a+2c+2d-16, 2d-1, 4a+2b-2d-12)|1 \leq d \leq b-2\}\cup\{\gamma_{E}(\eta|U_{E})=(4a+2c+2d-16, 2d-1, 2a+2b+2c-10)|d=b-1\}$.\\\\
For the edges $\{\eta=p_{2,2d}s_{2,2d-1}|1 \leq d \leq a-1\}$, the set of edge metric coordinates is as follow:\\
$PS_{2}=\{\gamma_{E}(\eta|U_{E})=(4b+2c+2d-13, 4a+2b-2d-13, 2d-1)|1 \leq d \leq a-2\}\cup\{\gamma_{E}(\eta|U_{E})=(2a+2b+2c-9, 4a+2b-2d-13, 2d-1)|d=a-1\}$.\\\\
For the edges $\{\eta=r_{2,2d}u_{2,2d-1}|1 \leq d \leq b-1\}$, the set of edge metric coordinates is as follow:\\
$RU_{2}=\{\gamma_{E}(\eta|U_{E})=(4c+2d-10, 2a+2b+2c-10, 4b-2d-8)|d=1\}\cup\{\gamma_{E}(\eta|U_{E})=(4c+2d-10, 4a+2b-2d-16, 4b-2d-8)|2 \leq d \leq b-1\}$.\\\\
For the edges $\{\eta=q_{2,2d}t_{2,2d-1}|1 \leq d \leq c-1\}$, the set of edge metric coordinates is as follows:\\
$QT_{2}=\{\gamma_{E}(\eta|U_{E})=(2d, 4a+2c+2d-15, 4c+2b-2d-11)|1\leq d \leq c-2\}\cup\{\gamma_{E}(\eta|U_{E})=(2d, 2a+2b+2c-8, 4c+2b-2d-11)|d=c-1 \}$.\\\\
For the edges $\{\eta=p_{3,2d-1}s_{3,2d-1}|1 \leq d \leq a-2\}$, the set of edge metric coordinates is as follow:\\
$PS_{3}=\{\gamma_{E}(\eta|U_{E})=(4a+2c-2d-15, 4b+2c+2d-16, 4a+2b-2d-16)|1 \leq d \leq a-2\}$.\\\\
For the edges $\{\eta=q_{3,2d-1}t_{3,2d-1}|1 \leq d \leq c-2\}$, the set of edge metric coordinates is as follows:\\
$QT_{3}=\{\gamma_{E}(\eta|U_{E})=(4a+2b+2d-18, 4b+2c-2d-14, 4a+2c-2d-15)|1\leq d \leq c-2\}$.\\\\
For the edges $\{\eta=r_{3,2d-1}u_{3,2d-1}|1 \leq d \leq b-2\}$, the set of edge metric coordinates is as follow:\\
$RU_{3}=\{\gamma_{E}(\eta|U_{E})=(4a+2b-2d-16, 2b+4c-2d-13, 4a+2b+2d-18)|1 \leq d \leq b-2\}$.\\\\
For the edges $\{\eta_{1}=t_{2,j-3}p_{3,i-5}, \eta_{2}=u_{2,2}s_{3,i-5}, \eta_{3}=s_{2,i-3}t_{3,j-5}, \eta_{4}=u_{1,k-3}q_{3,j-5}, \eta_{5}=t_{1,2}u_{3,k-5}, \eta_{6}=s_{1,i-3}r_{3,k-5}, \eta_{7}=p_{3,1}r_{3,1}, \eta_{8}=u_{3,1}q_{3,1}, \eta_{9}=s_{3,1}t_{3,1}\}$, the set of edge metric coordinates is as follow:\\
$V_{2}=\{\gamma_{E}(\eta_{1}|U_{E})=(2c-2, 2a+2b+2c-11, 2b-1), \gamma_{E}(\eta_{2}|U_{E})=(2c, 2a+2b+2c-12, 2b-3), \gamma_{E}(\eta_{3}|U_{E})=(2a+2b+2c-11, 2b-1, 2a-2), \gamma_{E}(\eta_{4}|U_{E})=(2a+2b+2c-12, 2b-3, 2a), \gamma_{E}(\eta_{5}|U_{E})=(2a-1, 2c-2, 2a+2b+2c-11), \gamma_{E}(\eta_{6}|U_{E})=(2a-3, 2c, 2a+2b+2c-12), \gamma_{E}(\eta_{7}|U_{E})=(2b+2b-8, 2b+2c-8, 2a+2b-7), \gamma_{E}(\eta_{8}|U_{E})=(2a+2b-7, 2b+2c-8, 2a+2b-6), \gamma_{E}(\eta_{9}|U_{E})=(2a+2c-6, 2b+2c-7, 2a+2b-8)\}$.\\\\

Now, from these sets of edge metric codes for the graph $FCS_{a,b,c}$, we find that $|P_{1}|=|P_{2}|=2a-2$, $|Q_{1}|=|Q_{2}|=2c-2$, $|R_{1}|=|R_{2}|=2b-2$, $|S_{1}|=|S_{2}|=2a-4$, $|T_{1}|=|T_{2}|=2c-4$, $|U_{1}|=|U_{2}|=2b-4$, $|P_{3}|=|S_{3}|=2a-6$, $|Q_{3}|=|T_{3}|=2c-6$, $|R_{3}|=|U_{3}|=2b-6$, $|PS_{1}|=|PS_{2}|=a-1$, $|RU_{1}|=|RU_{2}|=b-1$, $|QT_{1}|=|QT_{2}|=c-1$, $|PS_{3}|=a-2$, $|QT_{3}|=c-2$, $|RU_{3}|=b-2$, $|V_{1}|=12$, and $|V_{2}|=9$. We see that the sum of all of these cardinalities is equal to $|E(FCS_{a,b,c})|$ and which is $3(5a+5b+5c-21)$. Moreover, all of these sets are pairwise disjoint, which implies that $edim(FCS_{a,b,c})\leq 3$. To complete the proof, we have to show that $edim(FCS_{a,b,c})\geq 3$. To show this, we have to prove that there exists no edge metric generator $U_{E}$ for $FCS_{a,b,c}$ such that
$|U_{E}|\leq2$. Since, the graph $FCS_{a,b,c}$ is not a path graph, so the possibility of a singleton edge metric generator for $FCS_{a,b,c}$ is ruled out \cite{emd}. Next, suppose on the contrary that there exists an edge metric generator $U_{E}$ with $|U_{E}|=2$. Therefore, we have the following cases to be discussed (for the contradictions, the naturals $a$, $b$, and $c$ are $\geq5$):\\\\
\textbf{Case(\rom{1})} When $U_{E}=\{a, b\}$, where $a$ and $b$ are the vertices from the outer $pqr$-cycle of $FCS_{a,b,c}$.

\begin{itemize}
  \item Suppose $U_{E}=\{p_{1,1}, p_{1,d}\}$, $p_{1,d}$ ($2\leq d \leq i$). Then $\gamma(u_{3,k-5}r_{3,k-5}|U_{E})=\gamma(r_{3,k-5}r_{3,k-6} |U_{E})$, for $2\leq d\leq i$, a contradiction.
  \item Suppose $U_{E}=\{p_{1,1}, q_{1,d}\}$, $q_{1,d}$ ($1\leq d \leq j$). Then $\gamma(r_{1,2}u_{1,1}|U_{E})=\gamma(u_{1,1}u_{1,2} |U_{E})$, for $1\leq d\leq j-1$; $\gamma(p_{1,i}s_{1,i-2}|U_{E})=\gamma(s_{1,i-2}s_{1,i-3}|U_{E})$ when $d=j$, a contradiction.
  \item Suppose $U_{E}=\{p_{1,1}, r_{1,d}\}$, $r_{1,d}$ ($1\leq d \leq k$). Then $\gamma(p_{3,1}s_{3,1}|U_{E})=\gamma(p_{3,1}r_{3,1}|U_{E})$, for $1\leq d\leq k$, a contradiction.
  \item Suppose $U_{E}=\{p_{1,1}, p_{2,d}\}$, $p_{2,d}$ ($1\leq d \leq i$). Then $\gamma(t_{1,1}q_{3,1}|U_{E})=\gamma(q_{3,1}q_{3,2}|U_{E})$, for $1\leq d\leq i$, a contradiction.
  \item Suppose $U_{E}=\{p_{1,1}, r_{2,d}\}$, $r_{2,d}$ ($1\leq d \leq k$). Then $\gamma(s_{2,1}s_{2,2}|U_{E})=\gamma(s_{2,1}p_{2,2}|U_{E})$, for $1\leq d \leq k-1$; $\gamma(s_{1,1}t_{2,1}|U_{E})=\gamma(q_{2,2}t_{2,1}|U_{E})$, for $d=k$, a contradiction.
  \item Suppose $U_{E}=\{p_{1,1}, q_{2,d}\}$, $q_{2,d}$ ($1\leq d \leq j$). Then $\gamma(s_{2,1}s_{2,2}|U_{E})=\gamma(s_{2,1}p_{2,2}|U_{E})$, for $1\leq d\leq j$, a contradiction.
  \end{itemize}

\textbf{Case(\rom{2})} When $U_{E}=\{a, b\}$, where $a$ and $b$ are the vertices from the first interior cycle of $FCS_{a,b,c}$.

\begin{itemize}
  \item Suppose $U_{E}=\{s_{1,1}, s_{1,d}\}$, $s_{1,d}$ ($2\leq d \leq i-2$). Then $\gamma(t_{2,1}t_{2,2}|U_{E})=\gamma(t_{2,1}q_{2,2}|U_{E})$, for $2\leq d\leq i-1$, a contradiction.
  \item Suppose $U_{E}=\{s_{1,1}, r_{3,d}\}$, $r_{3,d}$ ($1\leq d \leq k-5$). Then $\gamma(q_{1,4}t_{1,3}|U_{E})=\gamma(t_{1,3}t_{1,4}|U_{E})$, for $1\leq d\leq k-5$, a contradiction.
  \item Suppose $U_{E}=\{s_{1,1}, t_{2,d}\}$, $t_{2,d}$ ($1\leq d \leq j-2$). Then $\gamma(p_{1,2}s_{1,1}|U_{E})=\gamma(s_{1,1}s_{1,2}|U_{E})$, for $1\leq d\leq j-2$, a contradiction.
  \item Suppose $U_{E}=\{s_{1,1}, p_{3,d}\}$, $p_{3,d}$ ($1\leq d \leq i-5$). Then $\gamma(q_{3,1}t_{3,1}|U_{E})=\gamma(t_{3,1}t_{3,2}|U_{E})$, for $1\leq d\leq 2a-3$, a contradiction.
  \end{itemize}

\textbf{Case(\rom{3})} When $U_{E}=\{a, b\}$, where $a$ and $b$ are the vertices from the second interior cycle of $FCS_{a,b,c}$.

\begin{itemize}
  \item Suppose $U_{E}=\{u_{1,1}, u_{1,d}\}$, $u_{1,d}$ ($1\leq d \leq k-2$). Then $\gamma(u_{1,1}r_{1,2}|U_{E})=\gamma(u_{1,1}t_{1,j-2}|U_{E})$, for $1\leq d\leq k-2$, a contradiction.
  \item Suppose $U_{E}=\{u_{1,1}, t_{1,d}\}$, $t_{1,d}$ ($1\leq d \leq j-2$). Then $\gamma(u_{1,1}u_{1,2}|U_{E})=\gamma(u_{1,1}r_{1,2}|U_{E})$, for $1\leq d\leq j-2$, a contradiction.
  \item Suppose $U_{E}=\{u_{1,1}, u_{3,d}\}$, $u_{3,d}$ ($1\leq d \leq k-5$). Then $\gamma(s_{1,i-4}p_{1,i-3}|U_{E})=\gamma(s_{1,i-4}s_{1,i-5}|U_{E})$, for $1\leq d\leq k-5$, a contradiction.
  \item Suppose $U_{E}=\{u_{1,1}, q_{3,d}\}$, $q_{3,d}$ ($1\leq d \leq j-5$). Then $\gamma(p_{3,1}s_{3,1}|U_{E})=\gamma(s_{3,1}s_{3,2}|U_{E})$, for $1\leq d\leq j-5$, a contradiction.
\end{itemize}

\textbf{Case(\rom{4})} When $U_{E}=\{a, b\}$, where $a$ and $b$ are the vertices from the third interior cycle of $FCS_{a,b,c}$.

\begin{itemize}
  \item Suppose $U_{E}=\{s_{2,1}, s_{2,d}\}$, $s_{2,d}$ ($2\leq d \leq i-2$). Then $\gamma(u_{2,k-2}u_{2,k-3}|U_{E})=\gamma(u_{2,k-2}r_{2,k-1}|U_{E})$, for $2\leq d\leq i-2$, a contradiction.
  \item Suppose $U_{E}=\{s_{2,1}, u_{2,d}\}$, $u_{2,d}$ ($1\leq d \leq k-2$). Then $\gamma(s_{2,1}s_{2,2}|U_{E})=\gamma(s_{2,1}p_{2,2}|U_{E})$, for $1\leq d\leq k-2$, a contradiction.
  \item Suppose $U_{E}=\{s_{2,1}, s_{3,d}\}$, $s_{3,d}$ ($1\leq d \leq i-5$). Then $\gamma(r_{1,k-3}u_{1,k-4}|U_{E})=\gamma(u_{1,k-4}u_{1,k-5}|U_{E})$, for $1\leq d\leq i-5$, a contradiction.
  \item Suppose $U_{E}=\{s_{2,1}, t_{3,d}\}$, $t_{3,d}$ ($1\leq d \leq j-5$). Then $\gamma(r_{1,k-3}u_{1,k-4}|U_{E})=\gamma(u_{1,k-4}u_{1,k-5}|U_{E})$, for $1\leq d\leq j-5$, a contradiction.
\end{itemize}

\textbf{Case(\rom{5})} When $U_{E}=\{a, b\}$, where $a$ is in outer $pqr$-cycle and $b$ is in the first interior cycle of $FCS_{a,b,c}$.

\begin{itemize}
  \item Suppose $U_{E}=\{p_{1,1}, s_{1,d}\}$, $s_{1,d}$ ($1\leq d \leq i-2$). Then $\gamma(t_{1,1}q_{1,2}|U_{E})=\gamma(t_{1,1}t_{1,2}|U_{E})$, for $1\leq d\leq i-2$, a contradiction.
  \item Suppose $U_{E}=\{p_{1,1}, t_{2,d}\}$, $t_{2,d}$ ($1\leq d \leq j-2$). Then $\gamma(u_{2,1}u_{2,2}|U_{E})=\gamma(u_{2,1}r_{2,2}|U_{E})$, for $1\leq d\leq j-2$, a contradiction.
  \item Suppose $U_{E}=\{p_{1,1}, r_{3,d}\}$, $r_{3,d}$ ($1\leq d \leq k-5$). Then $\gamma(q_{1,4}t_{1,3}|U_{E})=\gamma(t_{1,3}t_{1,4}|U_{E})$, for $1\leq d\leq k-5$, a contradiction.
  \item Suppose $U_{E}=\{p_{1,1}, p_{3,d}\}$, $p_{3,d}$ ($1\leq d \leq i-5$). Then $\gamma(q_{3,1}t_{3,1}|U_{E})=\gamma(t_{3,1}t_{3,2}|U_{E})$, for $1\leq d\leq i-5$, a contradiction.
\end{itemize}

\textbf{Case(\rom{6})} When $U_{E}=\{a, b\}$, where $a$ is in outer $pqr$-cycle and $b$ is in the second interior cycle of $FCS_{a,b,c}$.

\begin{itemize}
  \item Suppose $U_{E}=\{p_{1,1}, t_{1,d}\}$, $t_{1,d}$ ($1\leq d \leq j-2$). Then $\gamma(u_{1,1}u_{1,2}|U_{E})=\gamma(u_{1,1}r_{1,2}|U_{E})$, for $1\leq d\leq j-2$, a contradiction.
  \item Suppose $U_{E}=\{p_{1,1}, u_{1,d}\}$, $u_{1,d}$ ($1\leq d \leq k-2$). Then $\gamma(u_{1,k-2}s_{2,i-2}|U_{E})=\gamma(r_{1,k-2}u_{1,k-2}|U_{E})$, for $1\leq d\leq k-2$, a contradiction.
  \item Suppose $U_{E}=\{p_{1,1}, u_{3,d}\}$, $u_{3,d}$ ($1\leq d \leq k-5$). Then $\gamma(q_{1,4}t_{1,3}|U_{E})=\gamma(t_{1,3}t_{1,4}|U_{E})$, for $1\leq d\leq k-5$, a contradiction.
  \item Suppose $U_{E}=\{p_{1,1}, q_{3,d}\}$, $q_{3,d}$ ($1\leq d \leq j-5$). Then $\gamma(u_{1,k-2}u_{1,k-3}|U_{E})=\gamma(u_{1,k-3}u_{1,k-4}|U_{E})$, for $1\leq d\leq j-5$, a contradiction.
\end{itemize}

\textbf{Case(\rom{7})} When $U_{E}=\{a, b\}$, where $a$ is in outer $pqr$-cycle and $b$ is in the third interior cycle of $FCS_{a,b,c}$.

\begin{itemize}
  \item Suppose $U_{E}=\{p_{1,1}, u_{2,d}\}$, $u_{2,d}$ ($1\leq d \leq k-2$). Then $\gamma(s_{2,1}s_{2,2}|U_{E})=\gamma(s_{2,1}p_{2,2}|U_{E})$, for $1\leq d\leq k-2$, a contradiction.
  \item Suppose $U_{E}=\{p_{1,1}, s_{3,d}\}$, $s_{3,d}$ ($1\leq d \leq i-5$). Then $\gamma(q_{3,1}t_{3,1}|U_{E})=\gamma(t_{3,1}t_{3,2}|U_{E})$, for $1\leq d\leq i-5$, a contradiction.
  \item Suppose $U_{E}=\{p_{1,1}, s_{2,d}\}$, $s_{2,d}$ ($1\leq d \leq i-2$). Then $\gamma(r_{1,k}r_{1,k-1}|U_{E})=\gamma(r_{1,k-1}r_{1,k-2} |U_{E})$, for $1\leq d\leq i-2$, a contradiction.
  \item Suppose $U_{E}=\{p_{1,1}, t_{3,d}\}$, $t_{3,d}$ ($1\leq d \leq j-5$). Then $\gamma(r_{1,k}r_{1,k-1}|U_{E})=\gamma(r_{1,k-1}r_{1,k-2}|U_{E})$, for $1\leq d\leq j-5$, a contradiction.
\end{itemize}

\textbf{Case(\rom{8})} When $U_{E}=\{a, b\}$, where $a$ is in first interior cycle and $b$ is in the second interior cycle of $FCS_{a,b,c}$.

\begin{itemize}
  \item Suppose $U_{E}=\{s_{1,1}, t_{1,d}\}$, $t_{1,d}$ ($1\leq d \leq j-2$). Then $\gamma(q_{2,2}t_{2,1}|U_{E})=\gamma(t_{2,1}t_{2,2}|U_{E})$, for $1\leq d\leq j-2$, a contradiction.
  \item Suppose $U_{E}=\{s_{1,1}, u_{3,d}\}$, $u_{3,d}$ ($1\leq d \leq k-5$). Then $\gamma(q_{3,1}q_{3,2}|U_{E})=\gamma(q_{3,1}t_{3,1}|U_{E})$, for $1\leq d\leq k-5$, a contradiction.
  \item Suppose $U_{E}=\{s_{1,1}, u_{1,d}\}$, $u_{1,d}$ ($1\leq d \leq k-2$). Then $\gamma(r_{1,k-1}u_{1,k-2}|U_{E})=\gamma(u_{1,k-2}s_{2,i-2}|U_{E})$, for $1\leq d\leq k-2$, a contradiction.
  \item Suppose $U_{E}=\{s_{1,1}, q_{3,d}\}$, $q_{3,d}$ ($1\leq d \leq j-5$). Then $\gamma(u_{1,k-2}u_{1,k-3}|U_{E})=\gamma(u_{1,k-3}u_{1,k-4}|U_{E})$, for $1\leq d\leq j-5$, a contradiction.
\end{itemize}

\textbf{Case(\rom{9})} When $U_{E}=\{a, b\}$, where $a$ is in first interior cycle and $b$ is in the third interior cycle of $FCS_{a,b,c}$.

\begin{itemize}
  \item Suppose $U_{E}=\{s_{1,1}, u_{2,d}\}$, $u_{2,d}$ ($1\leq d \leq k-2$). Then $\gamma(p_{1,2}s_{1,1}|U_{E})=\gamma(s_{1,1}s_{1,3}|U_{E})$, for $1\leq d\leq k-2$, a contradiction.
  \item Suppose $U_{E}=\{s_{1,1}, s_{2,d}\}$, $s_{2,d}$ ($1\leq d \leq i-2$). Then $\gamma(r_{1,k}r_{1,k-1}|U_{E})=\gamma(r_{1,k-1}r_{1,k-2}|U_{E})$, for $1\leq d\leq i-2$, a contradiction.
  \item Suppose $U_{E}=\{s_{1,1}, t_{3,d}\}$, $t_{3,d}$ ($1\leq d \leq j-5$). Then $\gamma(r_{1,k}r_{1,k-1}|U_{E})=\gamma(r_{1,k-1}r_{1,k-2}|U_{E})$, for $1\leq d\leq j-5$, a contradiction.
  \item Suppose $U_{E}=\{s_{1,1}, s_{3,d}\}$, $s_{3,d}$ ($1\leq d \leq i-5$). Then $\gamma(q_{3,1}t_{3,1}|U_{E})=\gamma(t_{3,1}t_{3,2}|U_{E})$, for $1\leq d\leq i-5$, a contradiction.
\end{itemize}

\textbf{Case(\rom{10})} When $U_{E}=\{a, b\}$, where $a$ is in second interior cycle and $b$ is in the third interior cycle of $FCS_{a,b,c}$.

\begin{itemize}
  \item Suppose $U_{E}=\{u_{1,1}, u_{2,d}\}$, $u_{2,d}$ ($1\leq d \leq k-2$). Then $\gamma(q_{2,j}q_{2,j-1}|U_{E})=\gamma(q_{2,j-1}q_{2,j-2}|U_{E})$, for $1\leq d\leq k-2$, a contradiction.
  \item Suppose $U_{E}=\{u_{1,1}, s_{2,d}\}$, $s_{2,d}$ ($1\leq d \leq i-2$). Then $\gamma(r_{1,2}u_{1,1}|U_{E})=\gamma(u_{1,1}t_{1,j-2}|U_{E})$, for $1\leq d\leq i-2$, a contradiction.
  \item Suppose $U_{E}=\{u_{1,1}, t_{3,d}\}$, $t_{3,d}$ ($1\leq d \leq j-5$). Then $\gamma(s_{3,1}s_{3,2}|U_{E})=\gamma(p_{3,1}s_{3,1}|U_{E})$, for $1\leq d\leq j-5$, a contradiction.
  \item Suppose $U_{E}=\{u_{1,1}, s_{3,d}\}$, $s_{3,d}$ ($1\leq d \leq i-5$). Then $\gamma(u_{2,1}u_{2,d}|U_{E})=\gamma(u_{2,2}u_{2,3}|U_{E})$, for $1\leq d\leq i-5$, a contradiction.
\end{itemize}

As a result, we infer that for $FCS_{a,b,c}$, there is no edge resolving set $U_{E}$ such that $|U_{E}|=2$. Therefore, we must have $|U_{E}| \geq 3$ i.e., $edim(FCS_{a,b,c})\geq 3$. Hence, $edim(FCS_{a,b,c})=3$, which concludes the theorem.
\end{proof}
In terms of minimum IEMG, we have the following result

\begin{thm}
For $a,b,c\geq4$, the graph $FCS_{a,b,c}$ has an IEMG with cardinality three.
\end{thm}

\begin{proof}
To show that, for zigzag edge coronoid fused with starphene $FCS_{a,b,c}$, there exists an IEMG $U_{E}^{i}$ with $|U_{E}^{i}|=3$, we follow the same technique as used in Theorem $3$.\\\\
Suppose $U_{E}^{i} = \{p_{1,1}, r_{1,1}, r_{1, k}\} \subset V(FCS_{a,b,c})$. Now, by using the definition of an independent set and following the same pattern as used in Theorem $1$, it is simple to show that the set of vertices $U_{E}^{i}= \{p_{1,1}, r_{1,1}, r_{2, k}\}$ forms an IEMG for $FCS_{a,b,c}$ with $|U_{E}^{i}|=3$, which concludes the theorem. \\
\end{proof}

\section{Conclusions}
In this paper, we have studied the minimum vertex and edge metric generators for the zigzag edge coronoid fused with starphene $FCS_{a,b,c}$ structure. For positive integers $a,b,c\geq4$, we have proved that $dim(FCS_{a,b,c})=edim(FCS_{a,b,c})=3$ (a partial response to the question raised recently in \cite{emd}). We also observed that the vertex and edge metric generators for $FCS_{a,b,c}$ are independent. In future, we will try to obtain the other variants of metric dimension (for instance, fault-tolerant metric dimension (vertex and edge), mixed metric dimension, etc) for the graph $FCS_{a,b,c}$.

\end{document}